\documentclass[12pt]{article}
\usepackage{amssymb,amsthm,amsmath,amsfonts,latexsym,tikz,hyperref}
\usepackage[hmargin=.8in,vmargin=1in]{geometry}
\usepackage{tikz}

\newtheorem{thm}{Theorem}[section]
\newtheorem{prop}[thm]{Proposition}
\newtheorem{cor}[thm]{Corollary}
\newtheorem{lemma}[thm]{Lemma}

\newtheorem{example}[thm]{Example}

\DeclareMathOperator{\Des}{Des}
\DeclareMathOperator{\col}{col}
\DeclareMathOperator{\cone}{cn}

\DeclareMathOperator{\des}{des}

\DeclareMathOperator{\maj}{maj}
\DeclareMathOperator{\mmod}{mod}

\DeclareMathOperator{\supp}{supp}

\newcommand{\symm}{\mathfrak{S}}
\newcommand{\eps}{\epsilon}
\newcommand{\CC}{\mathbb{C}}
\newcommand{\RR}{\mathbb{R}}
\newcommand{\oeps}{\overline{\epsilon}}

\newcommand{\ZZ}{\mathbb{Z}}

\begin{document}

\title{A Polyhedral Proof of a Wreath Product Identity}
\author{Robert Davis and Bruce Sagan}

\maketitle

\begin{abstract}
	In 2013, Beck and Braun proved and generalized multiple identities involving permutation statistics via discrete geometry.
	Namely, they recognized the identities as specializations of integer point transform identities for certain polyhedral cones.
	They extended many of their proof techniques to obtain identities involving wreath products, but some identities were resistant to their proof attempts.
	In this article, we provide a geometric justification of one of these wreath product identities, which was first established by Biagioli and Zeng.
\end{abstract}
	
	
\section{Introduction}

Let $r,n \in \ZZ_{>0}$ and let $\ZZ_r$ and $\symm_n$ denote the cyclic group of order $r$ and symmetric group on $[n]=\{1,\dots,n\}$, respectively.
The \emph{wreath product of $\ZZ_r$ by $\symm_n$} is defined as
\[
	\ZZ_r \wr \symm_n = \{(\eps,\pi) \mid \eps \in [0,r-1]^n,\  \pi \in \symm_n\}
\]
where $[0,r-1]=\{0,1,\dots,r-1\}$.
It is often convenient to consider elements of $\ZZ_r \wr \symm_n$ as \emph{$r$-colored permutations}, and write them in the \emph{window notation}
\[
	(\eps,\pi) = [\pi(1)^{\eps_1}\, \dots \, \pi(n)^{\eps_n}].
\]
When $r=2$, the wreath product can be viewed as bijections $f$ on $\{\pm 1,\dots, \pm n\}$ such that $f(-i) = -f(i)$, called \emph{signed permutations}.
This set is also referred to as the \emph{Coxeter group of type $B_n$}, and as the \emph{hyperoctahedral group} because it is the symmetry group for the $n$-dimensional octahedron, that is, 
the $n$-dimensional cross-polytope. 
Since the $n$-dimensional cross-polytope is dual to the cube $[-1,1]^n$, the two have the same symmetry group.

In order to discuss certain statistics on wreath products, we define a partial order on the set $\{0^0\} \cup \{i^j \mid i \in [n],\, j \in [0,r-1]\}$.
First, we will use the convention that $\pi(0) = 0$ and $\eps_0 = 0$ for all $(\eps,\pi) \in \ZZ_r \wr \symm_n$.
We will insist that $j^{c_j} < k^{c_k}$ if one of the three following conditions hold:
\begin{enumerate}
	\item $j<k$ and $c_j = c_k = 0$; 
	\item $j > k$ and $c_j,c_k >0$;
	\item $j^{c_j} < 0^0$ for any $c_j > 0$.
\end{enumerate}
We will refer to this as the \emph{Biagioli-Zeng partial ordering}, or simply the \emph{BZ ordering}, on $\ZZ_r \wr \symm_n$.
In the case of $r=n=3$, the BZ ordering is
\[
	3^2, 3^1 < 2^2, 2^1 < 1^2, 1^1 < 0^0 < 1^0 < 2^0 < 3^0.
\]
This ordering allows us to define the \emph{descent set} of $(\eps,\pi) \in \ZZ_r \wr \symm_n$ by
\[
	\Des(\eps,\pi) = \{i \in [0,n-1] \mid \pi(i)^{\eps_i} > \pi(i+1)^{\eps_{i+1}}\}
\]
and the \emph{descent statistic} by $\des(\eps,\pi) = |\Des(\eps,\pi)|$.
In particular, we let $\Des(\pi)=\Des({\bf 0},\pi)$ where ${\bf 0}$ is the all-zero vector.  So $\Des(\pi)$ is the ordinary descent set of $\pi$ as an element of $\symm_n$.
An associated statistic is the \emph{major index}, defined as
\[
	\maj(\eps,\pi) = \sum_{i \in \Des(\eps,\pi)} i.
\]
We also define the \emph{color weight of} $(\eps,\pi) \in \ZZ_r \wr \symm_n$ to be 
\[
	\col(\eps,\pi) = \sum_{i=1}^n \eps_i.
\]
Since the color weight only depends on $\eps$, we can similarly define the \emph{color weight of} $\eps \in \ZZ_r^n$ to be
\[
	\col(\eps) = \sum_{i=1}^n \eps_i.
\]

In \cite{BeckBraun13}, the authors created vast generalizations of many identities involving statistics on the group of permutations, on signed permutations, and, to a more limited extent, on general wreath products.
Their techniques involved recognizing the identities as specializations of integer point generating functions for the cubes $[0,1]^n$ and $[-1,1]^n$.
Among the identities in which Beck and Braun were interested is the following $q,u$-analogue, originally proven by Biagioli and Zeng.
Recall that for a variable $q$ and nonnegative integer $n$,
\[
	[n]_q = \frac{1-q^n}{1-q}, 
\]
so that $[n]_q = 1 + q + \dots + q^{n-1}$ when $n > 0$ and $[0]_q = 0$.

\begin{thm}[Equation (8.1), \cite{BiagioliZengWreath}]\label{thm: BZ}
	For any positive integers $r, n$,
	\begin{equation}\label{eq: BZ}
		\sum_{k \geq 0} ([k+1]_q + u[r-1]_u[k]_q)^nt^k = \frac{\sum_{(\eps,\pi) \in \ZZ_r \wr \symm_n} q^{\maj(\eps,\pi)}t^{\des(\eps,\pi)}u^{\col(\eps,\pi)}}{\prod_{j=0}^n(1-q^jt)}.
	\end{equation}
\end{thm}

This identity resisted the proof techniques of Beck and Braun.
Nevertheless, a geometric proof does exist, and the purpose of this article is to provide one.

\section{Main Proof}

\subsection{Monomial Association}

Examining the left side of Equation~\eqref{eq: BZ} for a fixed $k$ leads one to believe that there is some connection with hypercubes.
Since there is a sum of $kr+1$ monomials being raised to the $n^{th}$ power, it suggests one should be able to associate the lattice points of $[0,kr]^n$ with the monomials of 
$([k+1]_q + u[r-1]_u[k]_q)^n$ in some manner.
After doing so, multiplying by $t^k$ and summing over all $k$ will result in bijectively associating each lattice point of the cone 
\[
	\cone([0,r]^n) = \{ (\alpha x,\alpha) \in \RR^{n+1} \mid x \in [0,r]^n, \alpha \in \RR\}
\]
with a monomial occurring on the left side of \eqref{eq: BZ}.
This cone can, of course, be constructed for any set $S \subset \RR^n$ by replacing $[0,r]^n$ with $S$. 

To simplify notation, let $C_{r,n} = \cone([0,r]^n)$, 
and let $e_i$ be the $i^{th}$ standard basis vector in ${\mathbb R}^{n+1}$.
For a particular nonnegative integer $k$, consider the points $(v,k) \in C_{r,n} \cap \ZZ^{n+1}$ lying on the hyperplane $x_{n+1} = k$.
For each $j = 0, \dots, kr$ and $i = 1,\dots,n$, set
\[
	m'(je_i,k) = \begin{cases}
				q^j & \text{ if } 0 \leq j \leq k, \\
				q^{(j-1) \mmod k}\ u^{\lfloor (j-1)/k\rfloor} & \text{ if } k < j \leq kr.
			\end{cases}
\]
where $(j-1) \mmod k$ is always chosen from $[0,k-1]$.
It follows easily from this definition that for any $i$ we have
\begin{equation}
\label{m'}
\sum_{j=0}^{kr} m'(je_i,k) =[k+1]_q + u[r-1]_u[k]_q.
\end{equation}
This association allows us to construct a map $m: C_{r,n} \cap \ZZ^{n+1} \to \CC[q,t,u]$ defined by
\[
	m(v_1,\dots,v_n,k) = \left(\prod_{i=1}^n m'(v_ie_i,k)\right)t^k.
\]
The cases of $r=n=2$ and $k \in \{1,2\}$ are depicted in Figure~\ref{fig: lattice pts}.

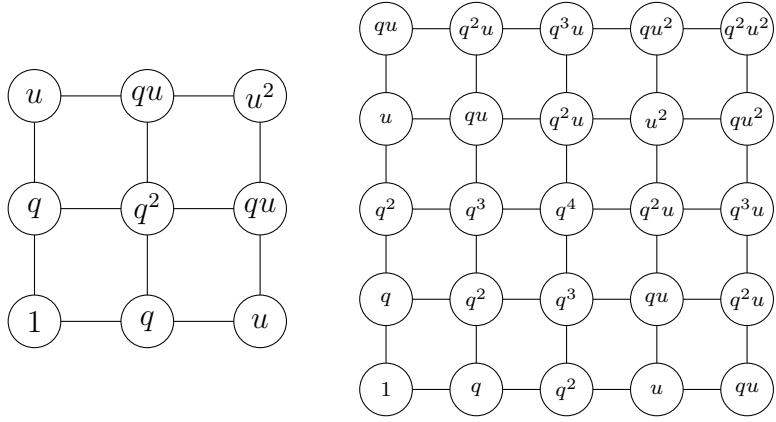
\begin{figure}
\begin{center}
	\begin{tikzpicture}[scale=1.5]
		\draw[step=1.0] (0,0) grid (2,2);
		
		\foreach \x in {0,1,2}
		{
			\foreach \y in {0,1,2} 
			{
				\node[draw,circle,fill=white,minimum size=0.7cm] at (\x,\y) {};
			}
		}
		
		\draw (0,0) node {$1$};
		\draw (1,0) node {$q$};
		\draw (2,0) node {$u$};
		\draw (0,1) node {$q$};
		\draw (0,2) node {$u$};
		\draw (1,1) node {$q^2$};
		\draw (1,2) node {$qu$};
		\draw (2,1) node {$qu$};
		\draw (2,2) node {$u^2$};
		
		\draw (0,-0.75) node {};
	\end{tikzpicture}
	\qquad
	\begin{tikzpicture}[scale=1.2]
		\draw[step=1.0] (0,0) grid (4,4);
		
		\foreach \x in {0,...,4}
		{
			\foreach \y in {0,...,4} 
			{
				\node[draw,circle,fill=white,minimum size=0.7cm] at (\x,\y) {};
			}
		}
	\draw (0,0) node {\scriptsize{$1$}};
		\draw (1,0) node {\scriptsize{$q$}};
		\draw (2,0) node {\scriptsize{$q^2$}};
		\draw (3,0) node {\scriptsize{$u$}};
		\draw (4,0) node {\scriptsize{$qu$}};
		\draw (0,1) node {\scriptsize{$q$}};
		\draw (0,2) node {\scriptsize{$q^2$}};
		\draw (0,3) node {\scriptsize{$u$}};
		\draw (0,4) node {\scriptsize{$qu$}};	
		\draw (1,1) node {\scriptsize{$q^2$}};
		\draw (2,1) node {\scriptsize{$q^3$}};
		\draw (3,1) node {\scriptsize{$qu$}};
		\draw (4,1) node {\scriptsize{$q^2u$}};
		\draw (1,2) node {\scriptsize{$q^3$}};
		\draw (2,2) node {\scriptsize{$q^4$}};
		\draw (3,2) node {\scriptsize{$q^2u$}};
		\draw (4,2) node {\scriptsize{$q^3u$}};
		\draw (1,3) node {\scriptsize{$qu$}};
		\draw (2,3) node {\scriptsize{$q^2u$}};	
		\draw (3,3) node {\scriptsize{$u^2$}};
		\draw (4,3) node {\scriptsize{$qu^2$}};
		\draw (1,4) node {\scriptsize{$q^2u$}};
		\draw (2,4) node {\scriptsize{$q^3u$}};
		\draw (3,4) node {\scriptsize{$qu^2$}};	
		\draw (4,4) node {\scriptsize{$q^2u^2$}};
	\end{tikzpicture}
\end{center}
	\caption{The monomials $m(v,k)$, with factors of $t$ divided out, associated to the lattice points in $([0,2]^2,1)$, left, and$([0,4]^2,2)$, right.}\label{fig: lattice pts}
\end{figure}

From the definition of $m$ and equation~\eqref{m'} we have, for fixed $k$,
\[
	\sum_{(v,k) \in ([0,kr]^n,k) \cap \ZZ^{n+1}} m(v,k) = ([k+1]_q + u[r-1]_u[k]_q)^nt^k.
\]
Summing over all $k$, we therefore have
\begin{equation}\label{eq: sum over cone}
	\sum_{k \geq 0} ([k+1]_q + u[r-1]_u[k]_q)^nt^k = \sum_{(v,k) \in C_{r,n} \cap \ZZ^{n+1}} m(v,k).
\end{equation}
This provides us with a geometric interpretation of the left side of \eqref{eq: BZ}.

\subsection{Cubical Decomposition}

Given $\eps = (\eps_1,\dots,\eps_n) \in \ZZ_r^n$, define the \emph{support} of $\eps$ to be the set
\[
	\supp(\eps) = \{i \mid \eps_i > 0\}.
\]
The \emph{cube associated to $\eps$} is
\[
	F_{\eps} = \left(\eps + [0,1]^n \right) \setminus \left(\bigcup_ {i \in \supp(\eps)} \{x \in \RR^n \mid x_i = \eps_i\}\right).
\]
So, if $\eps \neq 0$, then $F_{\eps}$ is partially open.
It is clear by construction that
\[
	[0,r]^n = \bigcup_{\eps \in \ZZ_r^n} F_{\eps}
\]
and that the union is disjoint.
See Figure~\ref{fig: associated cubes} for the decomposition for $n=r=2$.
This decomposition of the cube induces a decomposition of its cone, namely,
\begin{equation}\label{eq: disjoint union}
	C_{r,n} = \cone([0,r]^n) = \biguplus_{\eps \in \ZZ_r^n} \cone(F_{\eps}).
\end{equation}

\begin{figure}
\begin{center}
	\begin{tikzpicture}[scale=2]
		\draw[thick, fill = gray, fill opacity = 0.3] (0,0) -- (1,0) -- (1,1) -- (0,1) -- cycle;
		
		\draw[draw=none,fill = gray, fill opacity = 0.3] (1.2,0) -- (2.2,0) -- (2.2,1) -- (1.2,1) -- cycle;
		\draw[thick] (1.2,0) -- (2.2,0) -- (2.2,1) -- (1.2,1);
		\draw[dashed,thick] (1.2,0) -- (1.2,1);

		\draw[draw=none, fill = gray, fill opacity = 0.3] (0,1.2) -- (1,1.2) -- (1,2.2) -- (0,2.2) -- cycle;
		\draw[thick] (1,1.2) -- (1,2.2) -- (0,2.2) -- (0,1.2);
		\draw[dashed, thick] (1,1.2) -- (0,1.2);

		\draw[draw=none, fill = gray, fill opacity = 0.3] (1.2,1.2) -- (2.2,1.2) -- (2.2,2.2) -- (1.2,2.2) -- cycle;
		\draw[thick] (2.2,1.2) -- (2.2,2.2) -- (1.2,2.2);
		\draw[dashed, thick] (1.2,2.2) -- (1.2,1.2) -- (2.2,1.2);
		
		\foreach \x in {(0,0),(1,0),(0,1),(1,1),(2.2,0),(2.2,1),(0,2.2),(1,2.2),(2.2,2.2)}
		{
			\node[draw,circle,fill,inner sep = 1.5pt] at \x {};
		}
		\foreach \x in {(1.2,0),(1.2,1),(0,1.2),(1,1.2),(1.2,1.2),(1.2,2.2),(2.2,1.2)}
		{
			\node[draw,circle,fill=white,inner sep = 1.5pt] at \x {};
		}
		
		\node at (0.5,0.5) {$F_{(0,0)}$};
		\node at (0.5,1.7) {$F_{(0,1)}$};
		\node at (1.7,0.5) {$F_{(1,0)}$};
		\node at (1.7,1.7) {$F_{(1,1)}$};

	\end{tikzpicture}
\end{center}
	\caption{The four cubes associated to the elements $\eps \in \ZZ_2^2$ which decompose $[0,2]^2$.
}\label{fig: associated cubes}
\end{figure}
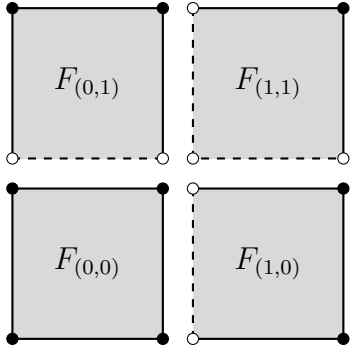

The cubes associated to $\eps \in \ZZ_r^n$ allow us to associate monomials according to which lattice points lie inside of $F_{\eps}$.
It will be easiest to obtain rational function expressions for
\[
	\sum_{(v,k) \in \cone(F_{\eps}) \cap \ZZ^{n+1}} m(v,k)
\]
by using relationships among the various $\eps \in \ZZ_r^n$.
First, we will see what happens when $\supp(\eps) = \supp(\eps')$.  

\begin{lemma}\label{lem: same support}
	If $\eps,\eps' \in \ZZ_r^n$ and $\supp(\eps) = \supp(\eps')$, then for any $\pi \in \symm_n$,
	\[
		\Des(\eps',\pi) = \Des(\eps,\pi).
	\]
	Consequently,
	\[
	\begin{aligned}
		\des(\eps',\pi) &= \des(\eps,\pi)\\
		\maj(\eps',\pi) &= \maj(\eps,\pi),		
	\end{aligned}
	\]
Also,
	\[
		\sum_{(v,k) \in \cone(F_{\eps}) \cap \ZZ^{n+1}} m(v,k) = \left(\sum_{(v',k) \in \cone(F_{\eps'}) \cap \ZZ^{n+1}} m(v',k)\right)u^{\col(\eps) - \col(\eps')}.
	\]
\end{lemma}

\begin{proof}
	Suppose $i \in \Des(\eps,\pi)$, so that $\pi(i)^{\eps_i} > \pi(i+1)^{\eps_{i+1}}$.
	This leads to three cases:
	\begin{enumerate}
		\item $\eps_i,\eps_{i+1} > 0$ and $\pi(i) < \pi(i+1)$;
		\item $0 = \eps_i < \eps_{i+1}$;
		\item $0 = \eps_i = \eps_{i+1}$ and $\pi(i) > \pi(i+1)$.
	\end{enumerate}
	In each case it is straightforward to verify that replacing $\eps_i$ with $\eps'_i$ and $\eps_{i+1}$ with $\eps'_{i+1}$ preserves the descent, so that $\Des(\eps,\pi) \subseteq \Des(\eps',\pi)$.
	The reverse containment follows by the same argument, swapping the roles of $\eps$ and $\eps'$.

 For the last equality of the lemma, since $\supp(\eps) = \supp(\eps')$ we have a bijection 
\[
\cone(F_{\eps'}) \cap \ZZ^{n+1}\rightarrow \cone(F_{\eps}) \cap \ZZ^{n+1}
\]
given by $(v',k)\mapsto(v,k)$ where $v=v' +k(\eps-\eps')$.  From the definition of $m'$ and $m$ we see that under this correspondence $m(v,k)=m(v',k) u^{\col(\eps) - \col(\eps')}$ and the summation follows.
\end{proof}

Next, we will obtain the desired rational function for special choices of $\eps$, where the previous lemma allows us to focus on such vectors which only contain zeros and ones.
For the proof, we will need to refine the cubes associated to the $\eps$.
Recall that the \emph{(type $A$) braid arrangement} is the hyperplane arrangement
\[
	\mathcal{A}_n = \bigcup_{1 \leq i < j \leq n} \{ (x_1,\dots,x_n) \in \RR^n \mid x_i = x_j \}.
\]
Its complement in $\RR^n$, that is, $\RR^n \setminus \mathcal{A}_n$, consists of the \emph{Weyl chambers}
\[
	W_{\pi} = \{x_{\pi(1)} > \dots > x_{\pi(n)}\},
\]
which range over all $\pi \in \symm_n$.
The intersections $W_{\pi} \cap [0,1]^n$ result in simplices, some of whose faces have been removed. The descent set of permutations can be used to decompose $[0,1]^n$ into the disjoint, partially-open simplices
\[
	\Delta_{\pi} = \{ x \in \RR_{\geq 0}^n \mid 1 \ge x_{\pi(1)} \ge \dots \ge x_{\pi(n)} \ge 0, \text{ and }  x_{\pi(i)} > x_{\pi(i+1)} \text{ if } i \in \Des(\pi)\}
\]
such that $\{\overline{\Delta}_{\pi}\}_{\pi \in \symm_n}$ is a unimodular triangulation of $[0,1]^n$.
These facts follow from arguments related to the theory of $P$-partitions; see \cite[Remark~4.2]{BeckBraun13} for a treatment of this connection.

It will be helpful to establish the following notation:
given $\eps = (\eps_1,\dots, \eps_n) \in \ZZ_r^n$, let
\[
\begin{aligned}
		G_{\eps} &= \{((\eps_{\pi(1)},\dots, \eps_{\pi(n)}),\pi) \in \ZZ_r \wr \symm_n \mid \pi \in \symm_n\} \\
				&= \{ [ \pi(1)^{\eps_{\pi(1)}}\, \dots\, \pi(n)^{\eps_{\pi(n)}}] \mid \pi \in \symm_n \}.
\end{aligned}
\]
In other words, the color associated to a letter stays fixed when ranging over all elements in $G_{\eps}$.

\begin{example}\label{ex: (1,0,1)}
Suppose $\eps = (1,0,1)$.
The set $G_{\eps}$ is
\[
	G_{\eps} = \{ [1^1\, 2^0\, 3^1], [1^1\, 3^1\, 2^0], [2^0\, 1^1\, 3^1], [2^0\, 3^1\, 1^1], [3^1\, 1^1\, 2^0], [3^1\, 2^0\, 1^1] \}.
\]
\[\begin{aligned}
	\Des [1^1\, 2^0\, 3^1] &= \{0,2\}, \\
	\Des [1^1\, 3^1\, 2^0] &= \{0,1\}, \\
	\Des [2^0\, 1^1\, 3^1] &= \{1,2\}, \\
	\Des [2^0\, 3^1\, 1^1] &= \{1\}, \\
	\Des [3^1\, 1^1\, 2^0] &= \{0\}, \\
	\Des [3^1\, 2^0\, 1^1] &= \{0,2\}.
\end{aligned}\]
\end{example}

We will now relate the sum of weights in a over $F_\eps$ to the generating function for $\maj$ and $\des$ over the corresponding $G_\eps$ for a certain $(0,1)$-vector $\eps$.
Following the proof will be an example that illustrates the techniques used.

\begin{prop}\label{prop: cone for few colors}
	Suppose $\eps = (1,\dots,1,0,\dots,0)$ where there are $l$ copies of $1$ so that $\col(\eps)=l$.
	Then
	\[
		\sum_{(v,k) \in \cone(F_{\eps}) \cap \ZZ^{n+1}} m(v,k) 
= \frac{\sum_{(\gamma,\pi) \in G_{\eps}} q^{\maj(\gamma,\pi)}t^{\des(\gamma,\pi)}u^{\col(\eps)}}{\prod_{j=0}^n (1-q^jt)}.
	\]
\end{prop}

\begin{proof}
	It is clear from the definition of $m$ that
	\begin{equation}\label{eq: m}
		\sum_{(v,k) \in \cone(F_{\eps}) \cap \ZZ^{n+1}} m(v,k) = u^l\sum_{k \geq 0} [k]_q^l[k+1]_q^{n-l}t^k.
	\end{equation}
	The coefficient of $q^rt^k$ in the above series therefore has the interpretation as the number of compositions $\alpha = (\alpha_1,\dots,\alpha_n)$ of $r$ (where parts equal to zero are allowed) such that
	the first $l$ parts are at most $k-1$ and the remaining $n-l$ parts are at most $k$.
	
	Now, note that the permutation
	\[
		\rho = l(l-1)\dots 21(l+1)(l+2)\dots n
	\]
	satisfies
	\[
		\rho\pi(i) = \begin{cases}
						l+1-\pi(i) & \text{ if } \pi(i) \leq l \\
						\pi(i) & \text{ if } \pi(i) > l
					\end{cases}
	\]
	for any $\pi \in \symm_n$.
	Based on the BZ ordering of the letters in $\ZZ_r \wr \symm_n$, we can conclude that for all $\pi \in \symm_n$, and setting $\gamma = (\eps_{\pi(1)},\dots,\eps_{\pi(n)})$,
	\[
		\Des(\gamma, \pi) = \begin{cases}
							\Des(\rho\pi) \uplus \{0\} & \text{ if } 0 \in \Des(\gamma,\pi) \\
							\Des(\rho\pi) & \text{ if } 0 \notin \Des(\gamma,\pi)
						\end{cases}.
	\]
	Recall that the set of all $(\gamma, \pi)$ constructed this way is exactly $G_{\eps}$.
	
	From the discussion following Lemma~\ref{lem: same support}, we know that for a particular composition $\alpha$ of the type we have been considering,
	there is exactly one $\pi \in \symm_n$ for which
	\begin{equation}\label{eq: inequalities}
		k- \eps_{\rho\pi(1)} \ge \alpha_{\rho\pi(1)} \ge \alpha_{\rho\pi(2)} \ge \dots \ge \alpha_{\rho\pi(n)} \ge 0,
	\end{equation}
	where $\alpha_{\rho\pi(i)} > \alpha_{\rho\pi(i+1)}$ if $i \in \Des(\rho\pi)$.	
	With this choice of $\pi$, we may construct a partition $\lambda = (\lambda_1,\dots,\lambda_n)$ of $r-\maj(\gamma,\pi)$ into $n$ parts (again allowing zeros), 
	with each part at most $k-\des(\gamma,\pi)$, by setting
	\[
		\lambda_i = \alpha_{\rho\pi(i)} - |\Des(\rho\pi) \setminus [i-1]|.
	\]
	Notice that for a fixed $\pi$, this map is a bijection between all compositions $\alpha$ satisfying \eqref{eq: inequalities} and partitions $\lambda$ of $r-\maj(\gamma,\pi)$ into $n$ parts, 
	each at most $k-\des(\gamma,\pi)$.

	Thus, the sum on the right-hand side of~\eqref{eq: m} can be computed by first computing  the generating function for the partitions corresponding to a fixed $\pi$, and then summing all of these generating functions.
	By our construction, the generating function for a fixed $\pi$ is
	\[
		\frac{q^{\maj(\gamma,\pi)}t^{\des(\gamma,\pi)}}{\prod_{j=0}^n (1-q^jt)}.
	\]
	Summing over all $(\gamma,\pi) \in G_{\eps}$ gives the claimed result.
\end{proof}

\begin{example}\label{ex: (1,1,0)}
Consider $\zeta = (1,1,0)$.
The set $G_{\zeta}$ is
\[
	G_{\zeta} = \{ [1^1\, 2^1\, 3^0], [1^1\, 3^0\, 2^1], [2^1\, 1^1\, 3^0], [2^1\, 3^0\, 1^1], [3^0\, 1^1\, 2^1], [3^0\, 2^1\, 1^1] \}.
\]
Notice that 
\[
\begin{aligned}
	\Des [1^1\, 2^1\, 3^0] &= \{0,1\}, \\
	\Des [1^1\, 3^0\, 2^1] &= \{0,2\}, \\
	\Des [2^1\, 1^1\, 3^0] &= \{0\}, \\
	\Des [2^1\, 3^0\, 1^1] &= \{0,2\}, \\
	\Des [3^0\, 1^1\, 2^1] &= \{1,2\}, \\
	\Des [3^0\, 2^1\, 1^1] &= \{1\}.
\end{aligned}
\]
By setting $\rho = 213$, then
\[
\begin{array}{rcl}
	\Des [\rho(1)^1\, \rho(2)^1\, \rho(3)^0] =& \Des(213) \uplus \{0\} &= \{0,1\}, \\
	\Des [\rho(1)^1\, \rho(3)^0\, \rho(2)^1] =& \Des(231) \uplus \{0\} &= \{0,2\}, \\
	\Des [\rho(2)^1\, \rho(1)^1\, \rho(3)^0] =& \Des(123) \uplus \{0\} &= \{0\}, \\
	\Des [\rho(2)^1\, \rho(3)^0\, \rho(1)^1] =& \Des(132) \uplus \{0\} &= \{0,2\}, \\
	\Des [\rho(3)^0\, \rho(1)^1\, \rho(2)^1] =& \Des(321) &= \{1,2\}, \\
	\Des [\rho(3)^0\, \rho(2)^1\, \rho(1)^1] =& \Des(312) &= \{1\}.
\end{array}
\]
\end{example}

We also need to consider the generating functions obtained by permuting the entries of $\eps$.

\begin{lemma}\label{lem: triple-preserving}
	Let $\eps, \eps' \in \ZZ_r^n$ such that $\eps'$ is a permutation of the entries of $\eps$.
	There is a $(\maj,\des,\col)$-preserving bijection between $G_{\eps}$ and $G_{\eps'}$.
	Consequently,
	\[
		\sum_{(\gamma,\pi) \in G_{\eps}} q^{\maj(\gamma,\pi)}t^{\des(\gamma,\pi)}u^{\col(\eps)}
= \sum_{(\gamma,\pi) \in G_{\eps'}} q^{\maj(\gamma,\pi)}t^{\des(\gamma,\pi)}u^{\col(\eps')}.
	\]
\end{lemma}

\begin{proof}
	Since the sets
	\[
		\{ 1^{\eps_1}, 2^{\eps_2}, \dots, n^{\eps_n}\}
	\]
	and
	\[
		\{ 1^{\eps'_1}, 2^{\eps'_2}, \dots, n^{\eps'_n}\}.
	\]
	are totally ordered, there is a unique order-preserving bijection $\omega$ between them.
	So, consider the map $\Omega:G_{\eps} \to G_{\eps'}$ defined by
	\[
		\Omega([\pi(1)^{\eps_1}\, \dots \, \pi(n)^{\eps_n}]) = [\omega(\pi(1)^{\eps_1}) \, \dots \, \omega(\pi(n)^{\eps_n})].
	\]
	Since $\eps$ and $\eps'$ are permutations of each other, $\Omega$ is $\col$-preserving.
	Moreover, $|\supp(\eps)| = |\supp(\eps')|$, and this together with the fact that $\omega$ is order-preserving proves that $\Omega$ is $\Des$-preserving, and therefore is $\des$- and $\maj$-preserving as desired.
\end{proof}

\begin{example}
Recall $\eps = (1,0,1)$ from Example~\ref{ex: (1,0,1)} and $\zeta = (1,1,0)$ from Example~\ref{ex: (1,1,0)}.
Since $3^1 < 1^1 < 2^0$ and $2^1 < 1^1 < 3^0$, we have
\[
	\begin{aligned}
		\omega(3^1) &= 2^1 \\
		\omega(1^1) &= 1^1 \\
		\omega(2^0) &= 3^0. 
	\end{aligned}
\]
Comparing the descent sets from Example~\ref{ex: (1,0,1)} and Example~\ref{ex: (1,1,0)}, we see that
\[
\begin{aligned}
	\Des \Omega([1^1\, 2^0\, 3^1]) &= \Des[1^1\, 3^0\, 2^1], \\
	\Des \Omega([1^1\, 3^1\, 2^0]) &= \Des[1^1\, 2^1\, 3^0], \\
	\Des \Omega([2^0\, 1^1\, 3^1]) &= \Des[3^0\, 1^1\, 2^1], \\
	\Des \Omega([2^0\, 3^1\, 1^1]) &= \Des[3^0\, 2^1\, 1^1], \\
	\Des \Omega([3^1\, 1^1\, 2^0]) &= \Des[2^1\, 1^1\, 3^0], \\
	\Des \Omega([3^1\, 2^0\, 1^1]) &= \Des[2^1\, 3^0\, 1^1].
\end{aligned}
\]
\end{example}

We can now combine our previous results to provide the penultimate step in our proof of Theorem~\ref{thm: BZ}.

\begin{cor}\label{cor: inner sum}
	For every $\eps \in \ZZ_r^n$,
	\[
		\sum_{(v,k) \in \cone(F_{\eps}) \cap \ZZ^{n+1}} m(v,k) = \frac{\sum_{(\gamma,\pi) \in G_{\eps}} q^{\maj(\gamma,\pi)}t^{\des(\gamma,\pi)}u^{\col(\eps)}}{\prod_{j=0}^n(1-q^jt)}.
	\]
\end{cor}

\begin{proof}
	Let $\eps' \in \ZZ_r^n$ be $\eps$ rewritten in weakly decreasing order, 
and let $\oeps = (1,\dots,1,0,\dots,0)$ where there are $|\supp(\eps)|$ copies of $1$.
	Using the definition of $m$ and then Lemma~\ref{lem: same support},
	\[
	\begin{aligned}
		\sum_{(v,k) \in \cone(F_{\eps}) \cap \ZZ^{n+1}} m(v,k) &= \sum_{(v,k) \in \cone(F_{\eps'}) \cap \ZZ^{n+1}} m(v,k) \\
				&= \left(\sum_{(v,k) \in \cone(F_{\oeps}) \cap \ZZ^{n+1}} m(v,k)\right)u^{\col(\eps') - \col(\oeps)}.
	\end{aligned}
	\]
	By Proposition~\ref{prop: cone for few colors},
	\[
		\sum_{(v,k) \in \cone(F_{\oeps}) \cap \ZZ^{n+1}} m(v,k) = 
			\frac{\sum_{(\gamma,\pi) \in G_{\oeps}} q^{\maj(\gamma,\pi)}t^{\des(\gamma,\pi)}u^{\col(\oeps)}}{\prod_{j=0}^n (1-q^jt)}.
	\]
	Plugging this last equality into the previous one and then using Lemmas~\ref{lem: same support} and \ref{lem: triple-preserving} in turn, we see
	\[
	\begin{aligned}
		\sum_{(v,k) \in \cone(F_{\eps}) \cap \ZZ^{n+1}} m(v,k)
		&=\left(\frac{\sum_{(\gamma,\pi) \in G_{\oeps}} q^{\maj(\gamma,\pi)}t^{\des(\gamma,\pi)}u^{\col(\oeps)}}{\prod_{j=0}^n (1-q^jt)}\right)u^{\col(\eps') - \col(\oeps)}\\
		&= \frac{\sum_{(\gamma,\pi) \in G_{\eps'}} q^{\maj(\gamma,\pi)}t^{\des(\gamma,\pi)}u^{\col(\eps')}}{\prod_{j=0}^n (1-q^jt)} \\
		&= \frac{\sum_{(\gamma,\pi) \in G_{\eps}} q^{\maj(\gamma,\pi)}t^{\des(\gamma,\pi)}u^{\col(\eps)}}{\prod_{j=0}^n (1-q^jt)},
	\end{aligned}
	\]
as desired.
\end{proof}

This bring us to the proof of the main theorem.

\begin{proof}[Geometric proof of Theorem~\ref{thm: BZ}]
	By Equations~\eqref{eq: sum over cone} and \eqref{eq: disjoint union}, the left hand side of \eqref{eq: BZ} is 
	\[
	\begin{aligned}
		\sum_{k \geq 0} ([k+1]_q + u[r-1]_u[k]_q)^nt^k  & = \sum_{(v,k) \in C_{r,n} \cap \ZZ^{n+1}} m(v,k) \\ 
					&= \sum_{\eps \in \ZZ_r^n} \sum_{(v,k) \in \cone(F_{\eps}) \cap \ZZ^{n+1}} m(v,k).
	\end{aligned}
	\]
	By Corollary~\ref{cor: inner sum}, we can express this as 
	\[
		\sum_{\eps \in \ZZ_r^n} \sum_{(v,k) \in \cone(F_{\eps}) \cap \ZZ^{n+1}} m(v,k) = 
			\sum_{\eps \in \ZZ_r^n} \frac{\sum_{(\gamma,\pi) \in G_{\eps}} q^{\maj(\gamma,\pi)}t^{\des(\gamma,\pi)}u^{\col(\eps)}}{\prod_{j = 0}^n(1-q^jt)}.
	\]
	Instead of expressing this as two separate sums, we can simply range over all elements of $\ZZ_r \wr \symm_n$.
	Recalling that $\col(\eps,\pi) = \col(\eps)$ for a given $\pi$, the right side of the above equation is 
	\[
		\frac{\sum_{(\eps,\pi) \in \ZZ_r \wr \symm_n} q^{\maj(\eps,\pi)}t^{\des(\eps,\pi)}u^{\col(\eps,\pi)}}{\prod_{j=0}^n(1-q^jt)},
	\]
	completing the proof.	
\end{proof}

\bibliographystyle{plain}
\bibliography{references}

\end{document}